\documentclass[a4paper, 11pt,reqno]{amsart}
\usepackage[utf8]{inputenc}
\usepackage{amsmath,amsthm,amsbsy,amssymb}
\usepackage{arydshln} 
\usepackage{booktabs} 
\usepackage{paralist} 
\numberwithin{equation}{section}
\makeatletter
\newcommand{\leqnos}{\tagsleft@true\let\veqno\@@leqno}
\newcommand{\reqnos}{\tagsleft@false\let\veqno\@@eqno}
\reqnos
\makeatother

\usepackage[hyphens]{url}

 \usepackage[backref=page]{hyperref}
\hypersetup{
    colorlinks = true,
linkcolor={red},
urlcolor={blue},
citecolor={green},    
urlcolor = {blue},
citebordercolor = {0.33 .58 0.33},
 linkbordercolor = {0.99 .28 0.23},
 breaklinks=true
}
 \def\UrlBreaks{\do\/\do-} 

\usepackage{mathrsfs}
\newcommand{\scr}[1]{\mathscr #1}

\newcommand{\scrB}{\mathscr B}
\newcommand{\scrC}{\mathscr C}

\newcommand{\scrG}{\mathscr G}

\newcommand{\cC}{\mathcal C}
\newcommand{\cD}{\mathcal D}

\newcommand{\fb}{\mathfrak b}
\newcommand{\fc}{\mathfrak c}
\newcommand{\fd}{\mathfrak d}
\newcommand{\fg}{\mathfrak g}
\newcommand{\fm}{\mathfrak m}

\newcommand{\fM}{\mathfrak M}

\usepackage{bm}
\newcommand*{\B}[1]{\ifmmode\bm{#1}\else\textbf{#1}\fi}
\newcommand{\bv}{\B{v}}
\newcommand{\bw}{\B{w}}

\newcommand{\bZ}{\B{0}}

\newcommand{\cW}{\mathcal W}

\newcommand{\RR}{\mathbb{R}}

\newcommand{\ZZ}{\mathbb{Z}}

\newcommand{\FF}{\mathbb{F}}

\newcommand{\sdfrac}[2]{\mbox{\small$\displaystyle\frac{#1}{#2}$}}

\usepackage{xspace}

\DeclareMathOperator{\dd}{d}

\DeclareMathOperator{\distance}{\mathfrak{d}} 
\DeclareMathOperator{\Spec}{\sigma}
\DeclareMathOperator{\spec}{\sigma}

\renewcommand{\pmod}[1]{\left( \mathrm{ mod\;}#1\right)}

\theoremstyle{plain}
\newtheorem{theorem}{Theorem}
\newtheorem{lemma}{Lemma}[section]
\newtheorem{corollary}{Corollary}

\newtheorem*{remark*}{Remark}

\newtheorem*{example*}{Example} 

\newtheorem{proposition}{Proposition}[subsection]

\usepackage{subfigure}  
\usepackage{float}
\usepackage[pdftex]{graphicx}
\graphicspath{{IMAGES/}{./}}
\usepackage[]{caption}
\usepackage{cite}

\begin{document}
\title[Visibility phenomena in hypercubes]
{Visibility phenomena in hypercubes}
\author[J. S. Athreya, C. Cobeli, A. Zaharescu]{Jayadev S. Athreya, Cristian Cobeli, 
Alexandru Zaharescu}

\address{
JA: Department of Mathematics, 
University of Washington, 
Padelford Hall, Seattle, WA 98195, USA
}
\email{jathreya@uw.edu}

\address{
CC: Simion Stoilow Institute of Mathematics of the Romanian Academy, 
P. O. Box 1-764, RO-014700 Bucharest, Romania}
\email{cristian.cobeli@gmail.com}

\address{
AZ: Department of Mathematics,
University of Illinois at Urbana-Champaign,
Altgeld Hall, 1409 W. Green Street,
Urbana, IL, 61801, USA and Simion Stoilow Institute of Mathematics of the Romanian Academy, 
P. O. Box 1-764, RO-014700 Bucharest, Romania}
\email{zaharesc@illinois.edu}  

\date{\today}






\begin{abstract}
We study the set of visible lattice points in multidimensional hypercubes.
The problems we  investigate mix together geometric, 
probabilistic and number theoretic tones.
For example, we prove that almost all self-visible triangles with vertices
in the lattice of points with integer coordinates in $\cW=[0,N]^d$
are almost equilateral having all sides almost equal 
to $\sqrt{d}N/\sqrt{6}$, and the sine of the typical angle between 
rays from the visual spectra from the origin of $\cW$ is, 
in the limit,  equal to $\sqrt{7}/4$, as $d$ and $N/d$ tend to infinity.
We also show that there exists an interesting number theoretic constant
$\Lambda_{d,K}$, which is the limit probability of the chance 
that a $K$-polytope 
with vertices in the lattice $\cW$ has all vertices visible from each other.

\end{abstract}
\maketitle

\section{Introduction}
Various phenomena related to distances in high dimensional spaces 
have attracted attention recently. For instance Gafni, Iosevich and Wyman~\cite{GIW2022} continue the study 
of the interesting connections with the unit distance problem in higher dimentions 
explored in~\cite{Ios2019, IMT2012, IRU2014, IS2016, OO2015}. 
Problems linked to the distribution of distances between points 
in finite sets placed 
in metric spaces (particular Euclidean spaces) have been investigated by
various authors from many different perspectives. A 
selection of such results, by no means complete, 
includes the works of general theoretical interest of
Bäsel~\cite{Bas2021},
Baileya, Borwein and Crandall~\cite{BBC2007},
Burgstaller and Pillichshammer~\cite{BP2009},
Dunbar~\cite{Dun1997},
Mathai,  Moschopoulos and Pederzoli~\cite{MMP1999}. 

There are also many practical applications of these problems, in particular in high-dimensional data analysis. For example, the article of Aggarwal, Hinneburg and Keim~\cite{AHK2001} 
is related to data mining techniques,
Li and Qiu~\cite{LQ2016} study probabilistic problems related to 
wireless communication networks,
Srinivasa and Haenggi~\cite{SH2010} are interested in
wireless networks whose efficiency is strongly influenced by the nodal distances, 
while Bubeck and Sellke~\cite{BS2021} prove a universal law of robustness that explains the necessity of overparametrization in deep neural networks.


In the present paper we study a few aspects related to visible lattice points in high dimensional hypercubes. Let $\cC:=[0,N]^d\subset \RR^d$ be the $d$ dimensional real cube of side length $N$, 
for some integers~\mbox{$d,N\ge 1$}.
Denote by $\cW:=\cC\cap \ZZ^d$ the set of $(N+1)^d$ points with integer coordinates in~$\cC$.
We denote by $\distance(\bv,\bw)$ the Euclidean distance 
between any two points $\bv=(v_1,\dots,v_d)$ and $\bw=(w_1,\dots,w_d)$.

The smallest distance between two points in $\cW$ is equal to~$1$, which is always met between 
two neighbor points, while the largest is attained by the opposite end points of 
the longest diagonals. Such points are $\bv=(0,\dots,0)$ and $\bw=(N,\dots,N)$ 
and the distance between them is $\distance(\bv,\bw)=\sqrt{d\cdot N^2}=Nd^{1/2}$. 
Then, it is natural to normalize $\distance(\bv,\bw)$ to obtain 
\textit{the normalized Euclidean distance} $\distance_d(\bv,\bw)$, for which all normalized 
distances between points in $\cW$ will belong to the interval $[0,1]$.
Thus, 
\begin{equation*}
  \distance_d(\bv,\bw):=\frac{1}{Nd^{1/2}}
    \bigg(\sum_{n=1}^d (w_n-v_n)^2\bigg)^{1/2}\,
\end{equation*}
and 
\begin{equation*}
  \distance(\bv,\bw)=Nd^{1/2}\distance_d(\bv,\bw) 
  = \bigg(\sum_{n=1}^d (w_n-v_n)^2\bigg)^{1/2}.
\end{equation*}
Denote by $\Omega\subset\cW\times\cW$ the set of pairs of points that 
are \textit{visible from each other}, that is, there are 
no other lattice points in $\cW$ between them on the straight line
segment that joins them. Then, by definition,
$\Omega$ is the set of all pairs $(\bv,\bw)\in\cW\times\cW$ such that 
\begin{equation}\label{eqgcd}
    \gcd(v_1-w_1,\dots,v_d-w_d)= 1.
\end{equation}
We show that the normalized distance between almost any two points in $\cW$ 
that are visible from each other is as close to
$1/\sqrt{6}\approx 0.40825$ as one wishes, if the dimension $d$ is sufficiently large and
$N$ is large enough with respect to $d$.

\begin{theorem}\label{Theorem1}
   For any $\varepsilon>0$, there exists an integer $C(\varepsilon)\ge 3$ 
   such that for any integers $d\ge C(\varepsilon)$ and 
   $N\ge C(\varepsilon)d$ we have:
\begin{equation}\label{eqTh1}
    \frac{1}{\#\Omega}\cdot \#\left\{(\bv,\bw)\in \Omega\; :\; 
        \distance_d(\bv,\bw)\in 
        \left[\sdfrac{1}{\sqrt 6}-\varepsilon, \sdfrac{1}{\sqrt 6}+\varepsilon\right]\right\} 
        \ge 1-\varepsilon\,.
\end{equation}
\end{theorem}
\noindent
As a consequence of Theorem~\ref{Theorem1} we see that
almost all triangles with vertices visible from each other are almost equilateral, 
almost all tetrahedrons with vertices visible from each other are 
almost regular and so on.  

In general, for any $K\ge 2$ let us denote by $\Omega_K$  
the set of $K$-polytopes with the property that any two of its vertices 
are visible from each other.
Said differently, if we call \textit{self-visible} 
a $K$-polytope with the property that from any of 
its vertices one can see all the others without any obstruction 
from any of the lattice points in~$\cW$, then 
\begin{equation*}
    \Omega_K :=\{P=(\bv_1,\dots,\bv_d)\in\cW^K \; :\; P \text{ is self-visible}\}.
\end{equation*}

Then, essentially, Theorem~\ref{Theorem1} can be restated in the following form, which is,
at the same time, a consequence and a more general form of it.

\begin{corollary}\label{Corollary1}
Let $K\ge 2$ be a fixed integer.
Then, for any 
$\varepsilon>0$, there exists an integer 
$C(K, \varepsilon)\ge 3$, such that for any integers $d\ge C(K,\varepsilon)$ and
\mbox{$N\ge C(K,\varepsilon)d$}, the proportion of polytopes $P\in \Omega_K$ for which
\begin{equation*}
    \distance_d(\bw',\bw'')\in 
        \left[1/\sqrt {6}-\varepsilon, 
        1/\sqrt {6} + \varepsilon\right]
\end{equation*}
for all distinct  $\bw', \bw''\in P$ is  greater than $1-\varepsilon$.
\end{corollary}


The next theorem answers the question of whether there is a limit
probability that a $K$-polytope in $\cW^K$ is self-visible.
\begin{theorem}\label{TheoremSelfVisible}
Let $d\ge 2$, $N\ge 2$ and $2\le K\le 2^d$ be integers.
Then, the probability that a $K$-polytope is self-visible is
\begin{equation}\label{eqProbabilityTHSV}
    \frac{\#\Omega_K}{\#\cW^{K}} = \prod_{\substack{\text{$p$ prime}}}
    \left(1-\frac{1}{p^d}\right)\cdots
    \left(1-\frac{K-1}{p^d}\right)
    +O\left(\frac{dK}{N^{1/2}}\right)
   + O\left(\frac{2^dK^2}{\log^{d-1}N}\right),
\end{equation}
and the implied constants in the big $O$ terms are absolute.
\end{theorem}
The infinite product over all primes in~\eqref{eqProbabilityTHSV}
is convergent and defines an endless square of constants 
\begin{equation*}
    \Lambda_{d,K}:=\prod_{\substack{\text{$p$ prime}}}\
    \prod_{k=1}^{K-1}
    \left(1-\frac{k}{p^d}\right),
\end{equation*}
which increase with $d$ and decrease with $K$.
They comprise some remarkable numbers for which Theorem~\ref{TheoremSelfVisible} gives a probabilistic geometric interpretation. 
For instance, if \mbox{$K=2$}, then  
$\Lambda_{2,2}=1/\zeta(2)=6/\pi^2\approx 0.6079271$ and
$\Lambda_{d,2}=\zeta(d)^{-1}$ for $d\ge 2$.
If $d=2$ and $K=3$, then 
\begin{equation*}
    \Lambda_{2,3}=\frac{6}{\pi^2}(2C_{FT}-1)\approx 0.196138,
\end{equation*}
where
\begin{equation*}
    C_{FT}:=\frac{1}{2}\left(1+\frac{1}{\zeta(2)}
\prod _{p}\left(1-{\frac{1}{p^2-1}}\right)\right)\approx 0.661317
\end{equation*}
is the Feller-Tornier constant 
(see Feller and Tornier~\cite{FT1932} and sequence A065493 from 
OEIS~\cite{OEISA065493}), which is related to the prime zeta function.

Let us note that if $N=0$ then $\cW^K$ and $\Omega_K$ are 
reduced to a single point. 
If $N=1$, all the points in $\cW$ are vertices. Then all
points in $\cW^K$ are self-visible, because there are no
intermediary points in the hypercube's lattice which blind 
the view from one point to another. Then, if $N=1$,
$\frac{\#\Omega_K}{\#\cW^{K}}=1$ for $2\le K\le 2^d$.

We also remark that for any $N>1$, if $K> 2^d$
then in~\eqref{eqProbabilityTHSV} both non-error terms, the one from the left 
side and the main term on the right side, are equal to zero.
In this case~\eqref{eqProbabilityTHSV} holds true with no error terms.

Let  $\vec{\bv}$ denote the ray that starts at the origin $\mathbf{0}=(0,\dots,0)\in\cW$ 
and passes through 
$\bv\in\cW$. Denote by $\Psi$ the set of pairs of rays $(\vec{\bv},\vec{\bw})$ with $(\bv,\bw)\in\Omega$.
The next result shows that if $d$ and $N/d$ are sufficiently large 
then almost all angles between rays from the origin
towards points that are visible to each other have the sine almost equal to $\sqrt{7}/4$. 

\begin{theorem}\label{Theorem2}
For any $\varepsilon>0$, there exists an integer $C(\varepsilon)\ge 3$ such that for 
any integers $d\ge C(\varepsilon)$ and $N\ge C(\varepsilon)d$ we have:
\begin{equation}\label{eqTh1Corollary2}
    \frac{1}{\#\Psi}\cdot \#\left\{(\vec{\bv},\vec{\bw})\in \Psi\; :\; 
        \sin\big(\widehat{\vec{\bv},\vec{\bw}}\big)\in 
        \left[\sdfrac{\sqrt 7}{4}-\varepsilon, \sdfrac{\sqrt 7}{4}+\varepsilon\right]
        \right\} 
        \ge 1-\varepsilon\,.
\end{equation}
\end{theorem}
Note that the statement in Theorem~\ref{Theorem2} does not depend on normalization, 
because the angles are preserved regardless of any scaling. 

For any two subsets $\scr{M}',\scr{M}''\subseteq\cW$, let $\spec(\scr{M}',\scr{M}'')$ 
be the visual \textit{spectrum}, 
which we define to be the set of sines of all angles between distinct rays 
that start from the origin towards the points in $\scr{M}'$ and 
$\scr{M}''$, that is, 
denoting identically a point $\fm$ and the ray from the origin towards $\fm$,   
\begin{equation}\label{eqSines}
    \spec(\scr{M}',\scr{M}'') := \left\{
    \sin(\fm',\fm'') \; :\; \fm'\in\scr{M}', \fm''\in\scr{M}'',\; \fm'\neq \fm''
    \right\}.
\end{equation}
If $\scr{M}'=\scr{M}''=\scr{M}$, we write shortly 
$ \spec(\scr{M})$ instead of $\spec(\scr{M},\scr{M})$.
The angles between rays from the origin to points in $\cW$ cover a larger
and larger set of possibilities as the dimension increases and
in the limit, as $d\to\infty$, there is a limit set of the spectrum
 $\spec(\scr{W})$, which is the interval~$[0,1]$. 
On top of that, choosing elements of $\Spec(\cW)$ is a random variable, 
which, by Theorem~\ref{Theorem2}, has a limit probability density function $f(t)$
that is discrete and concentrated in a single point, and
$f(t)=\delta\big(t-\frac{\sqrt{7}}{4}\big)$, for $0\le t\le 1$, where
$\delta$ is the Dirac distribution.

Since most points in $\cW$ are visible from each other, the
results in Theorems~\ref{Theorem1},\,\ref{Theorem2}  and Corollary~\ref{Corollary1} may prove useful to check particularities related to randomness of large set of data.
A related result about points on hyperspheres appearred 
in a theoretical context in the theory of neural network
(see Bubeck and Sellke~\cite{BS2021}).

\medskip
In different contexts in nature, it happens and it is not uncommon
for a property that is proved to be valid for almost overall objects in a certain universe to be difficult or even impossible to build or to indicate just a single instance that satisfy it. 
However, in the context of the hypercube lattice $\cW$,
we can extract some distinguished polytopes that offer a cross-section 
view of its inner structure.

Let $\scrC=\{\fc_0,\fc_1,\dots,\fc_d\}$ to be the set of points whose 
coordinates are the rows of the circular symmetric matrix
\begin{equation}\label{eqSymmetryMatrixC}
M_\scrC=
    \begin{bmatrix}
        0 & 1 & \cdots & d-1 & d\\
        1 & 2 & \cdots & d & 0\\
        2 & 3 & \cdots & 0 & 1\\
        \vdots & \vdots & \cdots & \vdots & \vdots\\
        d-1 & d & \cdots & d-3 & d-2\\
        d & 0 & \cdots & d-2 & d-1
    \end{bmatrix}
\end{equation}
and let $\scrG:=\{\fg_0,\dots,\fg_{p-1}\}$  be the set of points whose 
components are the rows of the matrix
\begin{equation}\label{eqSymmetryMatrixG}
M_\scrG=
    \begin{bmatrix}
        0 & 1^{-1} & \cdots & (p-2)^{-1} & (p-1)^{-1}\\
        1^{-1} & 2^{-1} & \cdots & (p-1)^{-1} & 0\\
        2^{-1} & 3^{-1} & \cdots & 0 & 1^{-1}\\
        \vdots & \vdots & \cdots & \vdots & \vdots\\
        (p-2)^{-1} & (p-1)^{-1} & \cdots & (p-4)^{-1} & (p-3)^{-1}\\
        (p-1)^{-1} & 0 & \cdots & (p-3)^{-1} & (p-2)^{-1}
    \end{bmatrix}
\end{equation}
Here $p$ is prime, the classes of the representatives of the inverses in $M_\scrG$
are taken from $\{1,\dots,p-1\}$ and for symmetry, by convention, 
we may set $0$ to be 'the inverse' of $0$.

\begin{theorem}\label{TheoremCG}
   Let $d\ge 2$ and $p$  be prime.
    Then we have:
    
    1. Any point in $\scrC\cup\scrG$ is visible from the origin.
    Any two points in $\scrG$ are visible from each other. If
    $d=p-1$ and $p$ is large enough any two points in $\scrC\cup\scrG$ are visible from each other.
    
    2. The limit set of the normalized distances between points in $\scrC$  
    is the interval $[0,1/2]$, as $d$  tends to infinity.
    
    3. The limit set of the normalized distances between points in $\scrG$  
    consists of the single point~$\{1/\sqrt{6}\}$, as $p$  tends to infinity.
    
    4. If $d=p-1$,  the limit set of the normalized distances between points in $\scrC$  and points in $\scrG$
    is also~$\{1/\sqrt{6}\}$, as $p$  tends to infinity.
    
    5. The limit of the spectrum $\spec(\scrC)$ is the interval $[0,\sqrt{39}/8]$, 
    as $d$  tends to infinity. 
    
    6. The limit of the spectrum $\spec(\scrG)$ consists of the single point $\{\sqrt{7}/4\}$, 
    as $p$  tends to infinity. 
    
    7. If $d=p-1$,  the limit of the spectrum 
    $\spec({\scrC,\scrG})$ consists of the single point~$\{\sqrt{7}/4\}$, also, 
     as $p$  tends to infinity.
\end{theorem}
Note the decimal approximation of the size of the spectra 
in Theorem~\ref{TheoremCG}:

$\sqrt{7}/4 \approx 0.661438$, 
$\arcsin(\sqrt{7}/4) \approx 0.72273$ radians or
$\approx 41.40962^\circ$; and

$\sqrt{39}/8 \approx 0.78063$, 
$\arcsin(\sqrt{39}/8) \approx 0.89566$ radians or
$\approx 51.31781^\circ$.

As one can see from Theorem~\ref{TheoremCG}, about a quarter of all distances 
\mbox{$\distance_d(\bv,\bw)$}
with $\bv,\bw\in\scrC\cup\scrG$ are singular, being different from the other three quarters 
which in the limit are all equal to $1/\sqrt{6}$.
However, in Subsection~\ref{SubsectionAverageCG} we prove that if $d=p-1$ then $A_p(\scrC\cup\scrG)$, 
the average of all normalized distances between points in $\scrC\cup\scrG$, is still the same
$1/\sqrt{6}$, in the limit as $p\to \infty$.


In Section~\ref{SubsectionMatrixB} we complement
the phenomena observed in Theorem~\ref{TheoremCG} 
with yet another polytope $\scrB$ whose vertices are visible from each other, even though they 
lie all almost  aligned on a straight line.
This set wraps around the diameter of $\cW$ being composed by close neighbors of 
the equally spaced points on the longest diagonal of the cube. 
The limit set of the normalized distances 
between the points in this example equals the full interval~$[0,1]$,
if their number tends to infinity.
The same goes for the limit of the spectrum~$\spec(\scrB)$, which is 
the interval~$[0,1]$, also.
Taken together, merged into a geometric spindle shape,
the polytopes $\scr C, \scr G$ and $\scr B$ combine their
arithmetic and probabilistic properties keeping in balance 
the spinning top intrinsic qualities of the still lattice 
hypercube $\cW$.

\smallskip

The paper is organized as follows. 
In Section~\ref{SectionCGD} we present the special polytopes
$\scrC, \scrG$ and~$\scrB$, checking the visibility and the mutual
distances between their vertices, which proves Theorem~\ref{TheoremCG}.
In Section~\ref{SectionOmega} we turn to the visibility
in the whole lattice and calulate the size of $\Omega$.
In  Section~\ref{SectionAv}
we find the average of the distances between points 
in $\cW$ that are visible from each other and in Section~\ref{SectionM2} we estimate 
the second moment about their mean.
We use these results in Section~\ref{SectionProofT1T2C1} to obtain
effective results that in particular prove 
Theorems~\ref{Theorem1}, \ref{Theorem2} and Corollary~\ref{Corollary1}.
In Section~\ref{SectionProbabilityOK} we discuss at large the problem of self-visible $K$-polytopes and prove Theorem~\ref{TheoremSelfVisible}.
We conclude in Section~\ref{SectionIntuition} 
with a possible good place to start. It is a short heuristics 
that might be useful to adjust with the little peculiarities 
that appear from higher dimensions. 
It is an intuitive touch on the matter, 
although it is done in a continuous, where visibility has no
meaning, unlike the discrete universe which we explore beyond.

\section{Three distinguished polytopes. Their edges and diagonals.}\label{SectionCGD}
Here we discuss three examples of polytopes with a number of vertices 
of order almost equal in size with the dimension of the hypercube.
The polytopes $\scrC$ and $\scrG$ whose vertices are the rows of 
the matrices $M_\scrC$ and 
$M_\scrG$ introduced by~\eqref{eqSymmetryMatrixC} and~\eqref{eqSymmetryMatrixG} 
are as similar in construction as they are very different in shape.
The first has the distances between its vertices well spread over a long interval, 
while the second has all the distances between the vertices approximately equal, 
being as `equilateral' as it could be, as $d\to\infty$.

\subsection{Proof of Theorem~\ref{TheoremCG} -- visibility}\rule[-10pt]{0pt}{10pt}\\
Notice first that by the definition of $\scrC$ and $\scrG$ and that
of visibility~\eqref{eqgcd}, all points in $\scrC\cup\scrG$ are visible 
from the origin.

\subsubsection{Any two points belonging to either $\scrC$ or $\scrG$  
are visible from each other}

Let us suppose~\mbox{$d =p-1$} and let $\bv,\bw\in \cW$. 
Also, suppose that the coordinates of $\bv$ are a
permutation of $\{0,1,\dots,d\}$ and the coordinates of $\bw$ 
are a circular rotation of the coordinates of $\bv$.
Then, if $\bv$ and $\bw$ were not visible from each other, then it would exist 
an integer $b\ge 2$ such that the following congruences would hold:
\begin{equation}\label{eqInvVisible}
    v_0-w_0\equiv 0\pmod b; \ v_1-w_1\equiv 0\pmod b;\,  \dots;\
    v_{d}-w_{d}\equiv 0\pmod b.
\end{equation}
But since $p$ is prime and $v_0,\dots,v_{d}$ are all distinct, 
as $w_0,\dots,w_{d}$ also are, 
and since they belong to the same set of numbers, $\{0,1,\dots, d\}$, 
which appear each in exactly two of the congruences in~\eqref{eqInvVisible}, 
we find that 
\begin{equation*}
    v_0\equiv v_1\equiv \cdots \equiv v_{d} \equiv r\pmod b,
\end{equation*}
for some $r\in\{0,1,\dots, b-1\}$. 
But this is impossible, unless $b=1$.
Therefore, two points that belong to one and the same set, 
be it either $\scrC$  or $\scrG$  are visible from each other.

\subsubsection{Any two points $\fc\in \scrC$ and $\fg\in \scrG$  
are visible from each other}

Suppose $d=p-1$  and let  $\fc\in \scrC$ and $\fg\in \scrG$.
Denote by $g = \gcd(v_1-w_1,\dots,v_d-w_d)$ the expression 
that needs to be checked in the condition~\eqref{eqgcd}. 
Since $g$ is invariant under the same circular rotation applied to both
$\fc\in \scrC$ and $\fg\in \scrG$, we may assume that
\begin{equation}\label{eqgc}
\arraycolsep=1.4pt\def\arraystretch{1.2}
\begin{array}{ccccccccl}
      \fg = (0,&1^{-1},&\dots,&(p-1-h)^{-1},&(p-h)^{-1},&(p+1-h)^{-1},&\dots, &(p-1)^{-1}&);\\
      \fc = (h,&h+1,&\dots,&p-1,&0,&1,&\dots, &h-1)\,,&
\end{array}
\end{equation}
for some $h\in\{0,1,\dots,p-1\}$. 
As usual, the coordinates are taken as their representatives modulo $p$ in the interval $[0,p-1]$.  
Since $1^{-1}\equiv 1\pmod p$ and $(p-1)^{-1}\equiv p-1\pmod p$,
the difference between the second components of $\fc$ and $\fg$ is 
$(h+1)-1^{-1}=h$ and 
the difference between the last components is $(h-1)-(p-1)^{-1}=h-p$.
Because $p$ is prime, these differences are relatively prime, so that $\fc$ and $\fg$ 
are visible from each other unless $h=0$.

Suppose now that $h=0$ in~\eqref{eqgc}.
Then, if $\fc$  is not visible from $\fg$, there exists a prime number $2\le q\le p$
such that $q \mid (n-n^{-1})$ for $1\le n\le p-1$. Let us notice that for any $0\le r \le p-1$
\begin{equation*}
    \#\left\{1\le n\le p-1 \; :\; \left| n - n^{-1}\right|=r\right\}\le 4.
\end{equation*}
This is because $n - n^{-1}\equiv \pm r\pmod p$ is equivalent to $n^2\mp rn-1\equiv 0\pmod p$, 
congruence that has at most two solutions for each sign. 
Even more precise, if $r\neq 0$, if the congruence $n - n^{-1}\equiv r\pmod p$ 
has solution $a$ then it has solution $-a^{-1}$, also. These solutions are always distinct unless
$a\equiv -a^{-1}\pmod p$, which happens only if $p\equiv 1\pmod 4$, in which case 
$a=\big(\frac{p-1}{2}\big)!$.
If $r=0$, there are exactly three values of $n\in\{0,1,\dots,p-1\}$ for which 
$n - n^{-1}\equiv  r\pmod p$, namely $0,1,p-1$.
In conclusion, putting together these observations, while counting separately in the cases 
$p=2$, $p\equiv 1\pmod 4$ and $p\equiv 3\pmod 4$, we obtain in all cases the same  
number of distinct absolute values of differences  
\begin{equation}\label{eqppe4p1}
    \#\left\{\left| n - n^{-1}\right| \; :\; 1\le n \le p-1\right\}=\left\lfloor \frac p4\right\rfloor+1.
\end{equation}
Then, since 
\begin{equation*}
    \left\{\left| n - n^{-1}\right| \; :\; 1\le n \le p-1\right\}\subset \{0,1,\dots, p-1\}
\end{equation*}
and by our assumption a prime~$q$ divides all differences $ n - n^{-1}$, 
it follows that $q$ has to be either $2$ or $3$.

If $q=2$, the equality~\eqref{eqppe4p1} says that the number of pairs $(n,n^{-1})$, $1\le n\le p-1$ 
of the same parity is $\lfloor p/4\rfloor+1$. But this is not in agreement with  
Lehmer's conjecture~\cite[Problem F12]{Guy2004}, which is proved also for shorter general 
arithmetic progressions\cite[Theorem 1]{CZ2001b}).
A~particular case of that result shows that if $I, J\subset\{1,\dots,p-1\}$ are arithmetic progressions of ratios
$d_1,d_2\ge 1$, then
\begin{equation}\label{eqLehmer}
    \#\left\{(a,b)\in I\times J \; :\; ab\equiv 1\pmod p\right\}=\frac{\#I\cdot \#J}{p} 
        +O\left(p^{1/2}\log^2 p\right).
\end{equation}
Then, if $d_1=d_2=2$, counting the pairs $(n,n^{-1})$ with either both even or both odd components, 
we see that their total number is $p/2+O\left(p^{1/2}\log^2 p\right)$, 
which contradicts~\eqref{eqppe4p1}.
Likewise, in the remaining case $q=3$, with $d_1=d_2=3$, counting the pairs $(n,n^{-1})$ 
whose components both give the same remainder $0$, $1$ or $2$ when dividing by $3$, 
we find that their total number is $p/3+O\left(p^{1/2}\log^2 p\right)$, which is
also different from~\eqref{eqppe4p1}.
In conclusion, $\fc$ and~$\fg$ are visible from each other, 
which concludes the proof of the first part of Theorem~\ref{TheoremCG}.

\vspace*{2mm}
\subsection{Proof of Theorem~\ref{TheoremCG}  -- the distances}\label{SubsectionDistances}
\subsubsection{Distances between points of $\scrC$}\label{SubsubsectionC}
Remark that the same rotation applied to the coordinates of two points, 
does not change the distance between them, which, in particular, shows that we have
\begin{equation}\label{eqSymmetry}
    \distance(\fc_k,\fc_l)=\distance(\fc_0,\fc_{l-k})\,, \text{ for any $0\le k\le l\le d$.}
\end{equation}
This means that the range of values of all distances between any two points in $\scrC$ is 
covered by the distances between $\fc_0$  and each of $\fc_1,\dots,\fc_{d}$.
A straightforward calculation shows that in closed form these are:
\begingroup
\setlength{\arraycolsep}{1pt}
    \begin{equation}\label{eqSymmetryAllDistancesA}
        \begin{array}{rl}
       \distance(\fc_0,\fc_1)&=\left((d-0)\cdot 1^2+1\cdot (d-0)^2\right)^{1/2}\\
       \distance(\fc_0,\fc_2)&=\left((d-1)\cdot 2^2+2\cdot (d-1)^2\right)^{1/2}\\
       \distance(\fc_0,\fc_3)&=\left((d-2)\cdot 3^2+3\cdot (d-2)^2\right)^{1/2}\\
        \hdotsfor{2}\\
        \distance(\fc_0,\fc_d)&=\left((d-(d-1))\cdot d^2+d\cdot (d-(d-1))^2\right)^{1/2}.
        \end{array}
\end{equation}
\endgroup
Not all of these numbers are distinct, because of the symmetry of the parabola: 
$x\mapsto f(x)$, where $f(x)=(d-(x-1))\cdot x^2+x\cdot (d-(x-1))^2=(d+1)x(d+1-x)$, for $x\in [0,d+1]$.
The maximum of $f(x)$  is attained for $x=(d+1)/2$ and it is equal to $(d+1)^3/4$.
Also, $f(x)=f(d+1-x)$ and the values of $f(x)$ on the integers 
between $0$ and $d+1$ cover quite uniformly  the interval $[0,(d+1)^3/4]$ as $d$ becomes sufficiently large.
Precisely, for any $y\in[0,1/2]$ and any $\varepsilon>0$, there are $\fc',\fc''\in\scrC$ such that 
$y-\varepsilon<\distance_d(\fc',\fc'')<y+\varepsilon$.
We summarize in the following proposition these remarks on the polytope~$\scrC$.

\begin{proposition}\label{PropositionC}
   Let $\scrC=\{\fc_0,\fc_1,\dots,\fc_d\}$ be the set of points 
   whose coordinates are the rows of matrix~\eqref{eqSymmetryMatrixC}. 
   Then, the set of normalized distances between any two points in $\scrC$ is equal to 
\begin{equation}\label{eqPropositionC}
   \cD(\scrC):=\left\{
    \frac{\sqrt{(d+1)x(d+1-x)}}{d^{3/2}} \; :\; 
         0\le x\le \left\lfloor\sdfrac{d+1}{2}\right\rfloor
   \right\}
\end{equation}
and the set $\cD(\scrC)$ is dense in the interval $[0,1/2]$ as $d$ tends to infinity.
\end{proposition}

\subsubsection{Distances between points of $\scrG$}\label{SubsubsectionG}

Let $p$ be a prime number and $N=d=p-1$. 
The polytope $\scrG=\{\fg_0,\dots,\fg_{p-1}\}$ is formally close to $\scrC$.
The components of the points are the same, except that the numbers are inverted modulo $p$. 
The classes of the representatives of the inverses are taken from $\{1,\dots,p-1\}$ and 
by convention the inverse of an integer divisible by $p$, which does not exist, 
is always replaced by $0$. 
Let us remark that the influence of just a single component in the first point, while 
the others are obtained by circular rotations as in $\scrG$, has small and even negligible 
influence as $p\to\infty$, on the mutual distances between the points in $\scrG$. 
This is why we could keep, for balance, in $\scrG$ the components zero, even if zero has no 
inverse modulo $p$.  

The main motivation for choosing $\scrG$ is the random spread of the inverses. 
Various ways to measure the randomness of inverses, triggered by~\cite[Problem F12]{Guy2004}, 
have been studied in~\cite{CGZ2001,
CZ2001b,
CZ2001a
}. In~\cite{Zah2003} the focus is on the values of polynomials and rational 
functions mod~$p$ and the results there might be also used to build 
other polytopes with similar characteristics. 

\begin{lemma}\label{LemmaCor2003}
If $p$ is prime and $1\le h\le p-1$, we have
    \begin{equation}\label{eqCor2003}
   \sum_{x\in\FF_p\setminus\{0,p-h\}}\left|(x+h)^{-1}-x^{-1}\right|^2
    = \frac{p^3}{6}+O\left(p^{5/2}\log^2 p\right).
\end{equation}
\end{lemma}
Here the inverses are calculated in $\FF_p$ and the absolute value calculates 
the distance between two natural numbers in $\{0,1,\dots,p-1\}$, the corresponding representatives of the residue classes \mbox{mod $p$} of the inverses.

For $h=1$, the estimate~\eqref{eqCor2003} is a particular case of~\cite[Corollary 1.3]{Zah2003}, which is proved using Weil's bounds~\cite{Wei1948} for exponential sums with rational functions. 
For $h\ge 2$ the estimate~\eqref{eqCor2003} can be proved in a similar manner.
  
Taking the square root and dividing by $p^{3/2}$, we find by Lemma~\ref{LemmaCor2003}, that 
the normalized distance between any two points $\fg',\fg''\in\scrG$~is
\begin{equation}\label{eqDistanceG}
   \distance_d(\fg',\fg'') = \sdfrac{1}{\sqrt{6}}\left(1+O\left(\sdfrac{\log^2 p}{\sqrt{p}}\right)\right),
\end{equation}
which proves part 3 of Theorem~\ref{TheoremCG}.


%
%
%

\subsubsection{Distances between $\fc\in\scrC$ and $\fg\in\scrG$}\label{SubsectionDistancesDG}

It suffices to find the distance between $\fc$ and $\fg$ in~\eqref{eqgc}.
For this we have to estimate the sums
\begin{equation}\label{eqSigma12}
   \distance^2(\fg,\fc) = \sum_{n=0}^{p-h-1}\left|(n+h)-n^{-1}\right|^2
     + \sum_{n=p-h}^{p-1}\left|(n+h-p)-n^{-1}\right|^2 = \Sigma_h'+\Sigma_h'',
\end{equation}
where $\Sigma_h'$ and $\Sigma_h''$ are the first and the second sum in~\eqref{eqSigma12}, respectively.
To calculate $\Sigma_h'$, let $L>1$ be fixed, denote $u=(p-h)/L$, $v=p/L$ 
and split the rectangle $[0,p-h]\times[0,p]$ into $L^2$ rectangles 
$T_{j,k}:=\big[ju,(j+1)u\big)\times\big[kv,(k+1v)\big)$.
Then
\begin{equation*}
\begin{split}
   \Sigma_h' & = \sum_{j=0}^{L-1}\sum_{k=0}^{L-1} \!\!\!
            \sum_{\ \ (n,n^{-1})\in T_{j,k}}\left(n+h-n^{-1}\right)^2.
\end{split}
\end{equation*}
Here the size of the summand can be kept under control, so that we can replace it
by its value on the lower left corner of $T_{j,k}$. Thus, on using~\eqref{eqLehmer} with
$I\times J=T_{j,k}$, we have
\begin{equation*}
\begin{split}
   \Sigma_h' & = \sum_{j=0}^{L-1}\sum_{k=0}^{L-1} 
        \left(ju+h-kv+O(p/L)\right)^2\left(\frac{uv}{p}+O\left(p^{1/2}\log^2 p\right)\right).
\end{split}
\end{equation*}
Now we factor the terms that do not depend on $j,k$ and expand the square
\begin{equation}\label{eqSigma113}
\begin{split}
   \Sigma_h' & = \left(\frac{uv}{p}+O\left(p^{1/2}\log^2 p\right)\right)
        \sum_{j=0}^{L-1}\sum_{k=0}^{L-1} g(j,k,u,v,p,L)
\end{split}
\end{equation}
where
\begin{equation*}
\begin{split}
     g(j,k,u,v,p,L)=&j^2u^2+h^2+k^2v^2+2juh-2jkuv-2hkv\\
        &+O(p^2/L^2)+O\Big(p(ju+h+kv)/L\Big).
\end{split}
\end{equation*}
Adding together the terms separately over $k$ and $j$, 
the sum of powers being denoted by \mbox{$S_r(M)=1^r+\cdots+M^r$}, and then collecting together 
the error terms, the double sum from~\eqref{eqSigma113} becomes
\begin{equation*}
\begin{split}
    \sum_{j=0}^{L-1}\sum_{k=0}^{L-1} g(j,k,u,v,p,L)=&  Lu^2S_2(L-1)+L^2h^2+Lv^2S_2(L-1)\\
      & + 2LuhS_1(L-1) -2uvS_1^2(L-1) -2LhvS_1(L-1)\\
      & + O(p^2) + O\Big(p\big(LuS_1(L)+hL^2+LvS_1(L)\big)/L\Big)\\
      =& \frac{L^4u^2}{3}+L^2h^2+\frac{L^4v^2}{3}
         + L^3uh - \frac{L^4uv}{2}-L^3vh
         + O\left(p^2L\right).
\end{split}
\end{equation*}
Next we insert this into~\eqref{eqSigma113}, replace $u,v$ by their definition and reduce the terms:
\begin{equation}\label{eqSigma114}
\begin{split}
   \Sigma_h' & = \Big(\sdfrac{uv}{p}+O\big(\sqrt{p}\log^2 p\big)\Big)
        \left(\sdfrac{L^4(u^2+v^2)}{3}+L^2h^2
         - \sdfrac{L^4uv}{2}-L^3h(v-u)
         + O\left(p^2L\right)\right)\\
         & =\Big(p-h+O\big(L^2\sqrt{p}\log^2 p\big)\Big)
         \left(\sdfrac{p^2+(p-h)^2}{3}
         - \sdfrac{p(p-h)}{2}
         + O\left(p^2/L\right)\right)\\
         & = \frac{(p-h)(p^2+2h^2-hp)}{6}+O(p^3/L)+O\left(L^2p^{5/2}\log^2 p\right).
\end{split}
\end{equation}

To estimate $\Sigma''_h$, we make the change of variables $m=p-n$. Note that the representative 
in the interval $[0,p-1]$ of the inverse of $m$  is $m^{-1}=p-n^{-1}$.
Then
\begin{equation*}
\begin{split}
   \Sigma_h'' & = \sum_{n=p-h}^{p-1}\left|(n+h-p)-n^{-1}\right)|^2
     =\sum_{m=1}^{h}\left|(m+p-h)-m^{-1}\right|^2.
\end{split}
\end{equation*}
Apart from the end limits of summation, this is exactly $\Sigma'_{p-h}$. Therefore we have
\begin{equation}\label{eqSigma21}
   \Sigma_h''  = \Sigma_{p-h}' + O(p^2).
\end{equation}
On combining the estimate~\eqref{eqSigma114} with~\eqref{eqSigma21} and inserting 
the results into~\eqref{eqSigma12}, we obtain
\begin{equation}\label{eqFinalDistanceCG}
\begin{split}
    \distance^2(\fg,\fc) & =   \frac{(p-h)(p^2+2h^2-hp)+h(2p^2+2h^2-3hp)}{6}
    +O\left(p^3/L+L^2p^{5/2}\log^2 p\right)\\
    & = \frac{p^3}{6} +O\left(p^3/L+L^2p^{5/2}\log^2 p\right).
\end{split}    
\end{equation}
Balancing the error terms, we find the optimal $L$ that we have fixed at the beginning, namely
$L = \left\lfloor p^{1/6}\log^{-2/3}p\right\rfloor$.
Thus we have proved the following result.
\begin{proposition}\label{PropositionCG}
   Let $d\ge 2$ and let $p$ be prime such that $d=p-1$. 
   Let \mbox{$\scrC=\{\fc_0,\fc_1,\dots,\fc_d\}$} 
   be the set of points 
   whose coordinates are the rows of matrix $M_\scr{C}$ from~\eqref{eqSymmetryMatrixC} and let
   $\scrG=\{\fg_0,\fg_1,\dots,\fg_d\}$ be the set of points 
   whose coordinates are the rows of matrix $M_\scr{G}$ from~\eqref{eqSymmetryMatrixG}. 
   Then, as $p$ tends to infinity, the limit set of the normalized distances between 
   any two points $\fc\in\scrC$ and $\fg\in\scrG$ is the single point $\{1/\sqrt 6\}$ 
   and 
\begin{equation}\label{eqDistanceCG}
   \distance_d(\fg,\fc) =\frac{1}{\sqrt{6}}+O\left(p^{-1/6}\log^{2/3} p\right).
\end{equation}
\end{proposition}
Proposition~\ref{PropositionCG} proves part 4 of Theorem~\ref{TheoremCG}.

\subsection{Proof of Theorem~\ref{TheoremCG} -- the spectra}\label{SubsectionSpectrumCG}\rule[-10pt]{0pt}{10pt}\\
%
From the origin, which we denote by $\bZ$, 
the normalized distances toward points in $\scrC$ or $\scrG$ are equal
\begin{equation}\label{eqDistancesCGZ}
\distance_d(\bZ,\fd_h)=\distance_d(\bZ,\fc_h)
    =\left(\frac{d(d+1)(2d+1)}{6}\right)^{1/2}d^{-3/2}
    = \frac{1}{\sqrt{3}}+O(1/d), 
\end{equation}
for $0\le h\le d$, if $p=d+1$. The distances between points in
$\scrC$ are given by Proposition~\ref{PropositionC}.
Since the set of distances between points in $\scrC$ are in the limit, 
as $d\to\infty$, dense in the interval $[0,1/2]$, the limit of the
spectrum $\spec(\scrC)$ is also a continuous interval. 
The end points of the limit $\spec(\scrC)$ come from the limit angles of the 
isosceles triangles with vertices in $\bZ$ having two edges equal to 
$1/\sqrt{3}+O(1/d)$, according to~\eqref{eqDistancesCGZ}, 
while the third edge  is the shortest and the longest distance between 
points in $\spec(\scrC)$, respectively.
By~\eqref{eqPropositionC}, the acutest triangle has its third edge equal to 
$O(1/d)$, while the triangle with the largest angle at $\bZ$ has its third
edge equal to $1/2+O(1/d)$. 
Then a straightforward calculation gives the end points of the limit spectrum 
$0$ and $\sqrt{39}/8$.

If $d=p-1$, the mutual distances between points either from 
$\scrG$ or from $\scrC\cup\scrG$ are all equal to $1/\sqrt{6}+O\left(p^{-1/6}\log^{2/3} p\right)$, by 
relation~\eqref{eqDistanceG} and 
Proposition~\ref{PropositionCG}, respectively.
Then, the limit spectra $\spec(\scrG)$ and $\spec(\scrC\cup\scrG))$
are equal and discrete, containing exactly one point.
According to~\eqref{eqDistancesCGZ}, the limit point is the sine of the acutest angle of the isosceles triangle with two edges equal to $1/\sqrt{3}$ and the third equal to 
$1/\sqrt{6}$. Since this is equal to~$\sqrt{7}/4$, 
these concludes the proof of the remaining parts of Theorem~\ref{TheoremCG}.

\smallskip
\subsection{The average distance between points 
in \texorpdfstring{$\scrC\cup\scrG$}{CuG}}\label{SubsectionAverageCG}\rule[-10pt]{0pt}{10pt}\\
%
Suppose $d=p-1$. 
The cardinality of $\scrC\cup\scrG$ is 
$4p^2$ and the average of the squares
of distances between its points is \mbox{$A(\scrC\cup\scrG)=T/(4p^2)$}, where $T$ is the following sum
\begin{equation}\label{eqAverageT123}
\begin{split}
    T=\sum_{(\fc',\fc'')\in\scrC\times\scrC}\distance^2(\fc',\fc'') 
    + \sum_{(\fg',\fg'')\in\scrG\times\scrG}\distance^2(\fg',\fg'') 
    + 2\sum_{(\fc,\fg)\in\scrC\times\scrG}\distance^2(\fc,\fg).
\end{split}    
\end{equation}
We denote by $T_1,T_2,T_3$ the sums on the right side of~\eqref{eqAverageT123}.
By~\eqref{eqSymmetryAllDistancesA}, the first sum is
\begin{equation}\label{eqAverageT1}
\begin{split}
    T_1
    = 2\sum_{h=1}^{p-1}(p-h)\cdot\big(ph(p-h)\big)
    =2p\big(pS_2(p-1)+S_3(p-1)\big)
    =\frac{p^5}{6}+O\left(p^4\right).
\end{split}    
\end{equation}
By~\eqref{eqDistanceG} and~\eqref{eqDistanceCG} 
the last two sums in~\eqref{eqAverageT123} together are
\begin{equation}\label{eqAverageT23}
\begin{split}
    T_2+T_3    = 6\sum_{h=1}^{p-1}
            h\left(\frac{p^3}{6}+O\left(p^{17/6}\log^{2/3} p\right)\right)
    = \frac{p^5}{2}+O\left(p^{29/6}\log^{2/3} p\right).
\end{split}    
\end{equation}
Now we can find the normalized distances between points in $\scrC\cup\scrG$ which is defined by 
\begin{equation}\label{eqAverageCGnormalizedDef}
\begin{split}
    A_p(\scrC\cup\scrG) := \sqrt{A(\scrC\cup\scrG)}/p^{3/2}.
\end{split}    
\end{equation}
Thus, on inserting~\eqref{eqAverageT23} and~\eqref{eqAverageT1} 
into~\eqref{eqAverageT123}, we find that~\eqref{eqAverageCGnormalizedDef} becomes
\begin{equation}\label{eqAverageCGnormalized}
\begin{split}
    A_p(\scrC\cup\scrG) &= \left(\frac{1}{4p^2}
    \bigg(\frac{p^5}{6}+\frac{p^5}{2}+O\left(p^{29/6}\log^{2/3} p\right)\bigg)\right)^{1/2} 
    p^{-3/2}\\
    & = \frac{1}{\sqrt{6}}+O\left(p^{-1/6}\log^{2/3} p\right).
\end{split}    
\end{equation}

\subsection{A polytope stretched out along the longest diagonal of \texorpdfstring{$\cW$}{W}}\label{SubsectionMatrixB}\rule[-10pt]{0pt}{10pt}\\
%
On the same theme, we construct a polytope $\scrB=\{\fb_0,\dots,\fb_d\}$ 
that is a cousin of $\scrC$ and $\scrG$. 
The polytope $\scrB$ streches along the diagonal $\big[(0,\dots,0);\; (d,\dots,d)\big]$, it has all vertices 
visible from each other and the normalized distances between them 
are dense in $[0,1]$, as the dimension $d$ tends to infinity.
All components of $\fb_h$ are set to be equal to $h$, except the $(h-1)$-th and 
$(h+1)$-th, which are equal to $h-1$ and $h+1$, respectively. 
Thus  
\begin{equation}\label{eqPointdh}
    \begin{array}{rlccccccccc}
      & \ \text{\scriptsize $0$-th} & &  &  &  & \text{\scriptsize $h$-th} &  &  &  & \text{\scriptsize $d$-th} \\    
\fb_h= & (\ h, & h, & \dots, & h, & h-1, & h, & h+1, & h, & \dots, & h\ ), 
    \end{array}
\end{equation}
for $h=0,1,\dots,d$, with the convention that $-1 = d$ and $d+1 = 0$.


If $d=3$ or $d\ge 5$ the points of $\scrB$ are visible from one another.
For points that are not too close this follows because most components of any point are equal, 
while the neighbors of just one component are the neighbor integers of the rank of the point.
For points that are near each other it can also be checked one by one that they are visible from each other.

The distances between the points of $\scrB$  are:
\begin{equation}\label{eqDistancePointsbjk}
   \distance_d(\fb_j,\fb_k) = d^{-3/2}\sqrt{d|k-j|^2 +O(1)}=\frac{|k-j|}{d}+O(1/d),
\end{equation}
for $1\le j,k,\le d$.
As a consequence,~\eqref{eqDistancePointsbjk} implies that the closure of the set of distances 
between points in $\scr{B}$ equals the full interval $[0,1]$, as $d$ tends to infinity.

Notice that all points in $\scrB$ are visible from the origin,
because~\eqref{eqPointdh} assures that condition~\eqref{eqgcd} is verified.

The distances from the origin to points in $\scrB$ are
\begin{equation}\label{eqDistancePoints0h}
   \distance_d(\bZ,\fb_h) =\frac{h}{d}+O(1/d), \text{ for $1\le h\le d$.}
\end{equation}

Since the limit set of the mutual distances between points in $\scrB$
is the full interval $[0,1]$, on combining~\eqref{eqDistancePointsbjk} 
and~\eqref{eqDistancePoints0h}, we see that there is a limit 
of the spectrum $\spec(\scrB)$, which is also a closed interval. 
Its smallest end point comes from the 
triangle with vertices $\bZ$, $\fb_{d-2}$ and~$\fb_{d-1}$,
and its largest from the 
triangle with vertices $\bZ$, $\fb_{1}$ and $\fb_{d-1}$.
Then the angle at~$\bZ$ of the first triangle tends to zero, and that 
of the second triangle tends to $\pi/2$, as $d\to\infty$.
As a consequence, the limit of the spectrum $\spec(\scrB)$ is the interval
$[0,1]$, as $d$ tends to infinity.

\section{The number of pairs in \texorpdfstring{$\Omega$}{Omega}}\label{SectionOmega}

Let $(\bv,\bw)\in \Omega$ and suppose
$\bv=(v_1,\dots,v_d)$ and  $\bw=(w_1,\dots,w_d)$. Then 
$\bv$ and $\bw$ are visible from each other. This means that
$ \gcd(v_1-w_1,\dots,v_d-w_d)= 1$. 
We can rewrite this condition by bringing M\"obius summation into play. Thus, we have
\begin{equation}\label{eqMBS}
    \sum_{
    \substack{1\le b \le N\\
    b|v_1-w_1\\
    \cdots\\
    b|v_d-w_d}}
    \mu(b)
    =\begin{cases}
    1 & \text{if $(\bv,\bw)\in\Omega$},\\[8pt]
    0 & \text{if $(\bv,\bw)\not\in\Omega$}.
    \end{cases}
\end{equation}
We start by finding an estimate for the cardinality of $\Omega$, which is the object of the following lemma.

\begin{lemma}\label{LemmaEstimateOmega}
There exists an absolute constant $C_0>0$, such that for all $d\ge 2$  and all $N\ge 3d$, we have
\begin{equation}\label{eqEstimateOmega}
    \left|
    \#\Omega - \frac{N^{2d}}{\zeta(d)}
    \right|\le
    \begin{cases}
       C_0N^{3}\log N & \text{ if $d=2$}, \\[8pt]
       C_0dN^{2d-1} & \text{ if $d\ge 3$}.
    \end{cases}
\end{equation}
\end{lemma}
\begin{proof}
By the definition and the counting formula~\eqref{eqMBS}, by changing the order of summation, we have
\begin{equation*}
    \begin{split}
    \#\Omega&=  \sum_{\bv\in\cW}\sum_{\substack{\bw\in\cW\\(\bv,\bw)\in\Omega}}1 
    =  \sum_{\bv\in\cW}\sum_{\bw\in\cW}\sum_{\substack{b|(v_1-w_1)\\\cdots\\b|(v_d-w_d)}}\mu(d) 
    = \sum_{b=1}^N\mu(b)
        \sum_{\substack{0\le v_1, w_1\le N \\b| v_1-w_1}}\cdots
        \sum_{\substack{0\le v_d, w_d\le N\\b| v_d-w_d}} 1.
    \end{split}
\end{equation*}    
Since the variables run independently, the summation over $\bv$ and $\bw$ 
can be grouped as a product as follows
\begin{equation}\label{eqO2}
    \begin{split}
        \#\Omega
        &=\sum_{b=1}^N\mu(b)\bigg(\sum_{\substack{0\le v_1, w_1\le N\\b| v_1-w_1}} 1 
        \bigg)\cdots
        \bigg(\sum_{\substack{0\le v_d, w_d\le N\\b| v_d-w_d}} 1 
        \bigg)\,.
    \end{split}
\end{equation}
Next, we estimate the inner sums that are equal to each other for all $j\in\{1,\dots,d\}$. 
Dropping the subscripts, we see that each of them is equal to
\begin{equation*}
    \begin{split}
     H:=\sum_{\substack{0\le v, w\le N\\b| v-w}}1&=
     \sum_{0\le r\le b-1}\#\{
     (v,w) : 0\le v,w\le N, v\equiv r\pmod b, w\equiv r\pmod b
     \}\\
     & = \sum_{0\le r\le b-1}\big(
     \#\{0\le a\le N \; :\; a\equiv r\pmod b\}
     \big)^2.
    \end{split}
\end{equation*}
The cardinality of the inner set is equal to 
$\left\lfloor\frac{N-r}{b}\right\rfloor+1=\frac{N}{b}+\theta_1(b,r,N)$, with $|\theta(b,r,N)|\le 1$.
Then
\begin{equation*}
   H =  \sum_{0\le r\le b-1}\left(\left\lfloor\frac{N-r}{b}\right\rfloor+1    \right)^2
    =  \sum_{0\le r\le b-1}\left( \frac{N^2}{b^2} + \theta_2(b,r,N)\cdot \frac{N}{b}\right) 
    = \frac{N^2}{b} + \theta(b,N)N,
\end{equation*}
for some real numbers for which $|\theta_2(b,r,N)|\le 3$ and $|\theta(b,N)|\le 3.$ 
On inserting this estimate in~\eqref{eqO2}, it yields
\begin{equation}\label{eqO3}
    \begin{split}
        \#\Omega&=\sum_{b=1}^N\mu(b)H^d
        =\sum_{b=1}^N\mu(b)\left(\frac{N^{2d}}{b^{d}} 
        + O\left(\sum_{k=1}^d\binom{d}{k}\bigg(\frac{N^2}{b}\bigg)^{d-k}(3N)^k\right)\right).
    \end{split}
\end{equation}
Here, the main term is
\begin{equation}\label{eqMTO4}
    \begin{split}
        \sum_{b=1}^N\mu(b)\frac{N^{2d}}{b^{d}}&=
        N^{2d}\sum_{b=1}^\infty\mu(b)b^{-d}+O\left(N^{2d}\int_{N}^\infty x^{-d}\dd x \right)
        =\frac{N^{2d}}{\zeta(d)}+O\left(\frac{N^{d+1}}{d} \right).
    \end{split}
\end{equation}
Denoting the error term in~\eqref{eqO3} by $E_1$ and changing the order of summation we find that
\begin{equation}\label{eqE11}
    \begin{split}
        |E_1|
        &= O\left(\sum_{b=1}^N\sum_{k=1}^d\binom{d}{k}\frac{N^{2d-2k}}{b^{d-k}}3^kN^k\right)
        =  O\left(\sum_{k=1}^d 3^kN^{2d-k}\binom{d}{k}T_{d-k}(N)\right),
    \end{split}
\end{equation}
where \mbox{$T_r(N):= 1^{-r}+\cdots+N^{-r}$}. 
Then $T_{d-k}(N)=N$, if $k=d$, $T_{d-k}(N)=\log N$, if $k=d-1$ and $T_{d-k}(N)=O(1)$ if $1\le k\le d-2$.
Then
\begin{equation}\label{eqE12}
    \begin{split}
        |E_1| & = O\left(3^dN^{d+1}\right)  + O\left(3^dN^{d+1}d\log N\right)+
         O\left(N^{2d}\sum_{k=1}^{d-2} 3^kN^{-k}\binom{d}{k}\right),
    \end{split}
\end{equation}

The lemma follows by inserting the estimates~\eqref{eqMTO4} and~\eqref{eqE12} into~\eqref{eqO3}.

\end{proof}

\section{The average distance between points visible from each other}\label{SectionAv}
The average of the square of distances between points visible from each other is
\begin{equation}\label{eqAv110}
    \begin{split}
        A_{vis}(d,N)&:=\frac{1}{\#\Omega}
    \sum_{(\bv,\bw)\in\Omega}\distance^2(\bv,\bw).
    \end{split}
\end{equation}
By~\eqref{eqMBS}, changing the order summation this is
\begin{equation}\label{eqAv33}
    \begin{split}
        A_{vis}(d,N)&=\frac{1}{\#\Omega}\sum_{b=1}^N\mu(b)
    \sum_{\bv\in\cW}\sum_{\substack{\bw\in\cW\\b|(v_1-w_1)\\
        \cdots\\
    b|(v_d-w_d)
    }}\distance^2(\bv,\bw).
    \end{split}
\end{equation}
We rewrite~\eqref{eqAv33} as
\begin{equation}\label{eqAv11}
    \begin{split}
      A_{vis}(d,N)&=\frac{1}{\#\Omega}
      \sum_{b=1}^N
      \mu(b)
    \sum_{\bv\in\cW}\sum_{\substack{\bw\in\cW\\b|(v_1-w_1)\\
        \cdots\\
    b|(v_d-w_d)
    }}\left((w_1-v_1)^2+\cdots+(w_d-v_d)^2\right)
    =\sum_{j=1}^d H_j,
    \end{split}
\end{equation} 
where
\begin{equation}\label{eqAv12}
    \begin{split}
      H_j&:=\frac{1}{\#\Omega}
      \sum_{b=1}^N
      \mu(b)
    \sum_{\bv\in\cW}\sum_{\substack{\bw\in\cW\\b|(v_1-w_1)\\
        \cdots\\
    b|(v_d-w_d)
    }}(w_j-v_j)^2.
    \end{split}
\end{equation} 
By changing the order of summation to isolate the part that does not depend on the $j$ variables, 
$H_j$ can be rewritten as 
\begin{equation}\label{eqAv13}
    \begin{split}
      H_j&=\frac{1}{\#\Omega}
      \sum_{b=1}^N       \mu(b)\bigg(
          \sum_{\substack{0\le v_j,w_j\le N\\b|(v_j-w_j)     }}(w_j-v_j)^2\bigg)
    \prod_{\substack{k=1\\k\not=j}}^d
           \bigg(\sum_{\substack{0\le v_k,w_k\le N\\b|(v_k-w_k)     }} 1\bigg)\\
    &=\frac{1}{\#\Omega}
      \sum_{b=1}^N       \mu(b)
      \left(\frac{N^2}{b}+O(N)\right)^{d-1}
      \sum_{\substack{0\le v_j,w_j\le N\\b|(v_j-w_j)     }}(w_j-v_j)^2.
    \end{split}
\end{equation} 
In the interior sum from~\eqref{eqAv13} we group the terms with $v_j$ and $w_j$ 
in the same residue classes mod $b$ as follows:
\begin{equation}\label{eqAv14}
    \begin{split}
      \sum_{\substack{0\le v_j,w_j\le N\\b|(v_j-w_j)     }}(w_j-v_j)^2
      &= \sum_{r=0}^{b-1}
        \sum_{\substack{0\le v_j,w_j\le N\\v_j\equiv w_j\equiv r \pmod{b}}}(w_j-v_j)^2\\
      &= \sum_{r=0}^{b-1}
        \sum_{m=0}^{\left\lfloor\frac{N-r}{b}\right\rfloor}
        \sum_{n=0}^{\left\lfloor\frac{N-r}{b}\right\rfloor}
        \big((r+mb)-(r+nb)\big)^2\\
       &= b^2\sum_{r=0}^{b-1} 
        \sum_{m=0}^{\left\lfloor\frac{N-r}{b}\right\rfloor}
        \sum_{n=0}^{\left\lfloor\frac{N-r}{b}\right\rfloor}
        (m-n)^2.
    \end{split}
\end{equation} 
With $M=\left\lfloor\frac{N-r}{b}\right\rfloor$, 
the sums over $m$ and $n$ are equal to 
\begin{equation}\label{eqAv15}
    \begin{split}
    \sum_{m=0}^{M} \sum_{n=0}^{M} (m-n)^2
      & = 2\sum_{m=0}^{M}\sum_{n=0}^{M} m^2 -2 \sum_{m=0}^{M}\sum_{n=0}^{M}mn\\
        & = 2\left(\frac{1}{3}M^4 + O(M^3)\right) -2\left(\frac{1}{4}M^4 + O(M^3)\right)\\
        & = \left(\frac{1}{6}M^4 + O(M^3)\right).
    \end{split}
\end{equation} 
Combining~\eqref{eqAv15} into~\eqref{eqAv14} we find that
\begin{equation}\label{eqAv155}
    \begin{split}
    \sum_{\substack{0\le v_j,w_j\le N\\b|(v_j-w_j)     }}(w_j-v_j)^2
      &=\frac{N^4}{6b} + O(N^3).
    \end{split}
\end{equation}
On inserting this estimate in~\eqref{eqAv13}, we obtain
\begin{equation}\label{eqAv16}
    \begin{split}
      H_j
    &=\frac{1}{\#\Omega}
      \sum_{b=1}^N       \mu(b)
      \left(\frac{N^2}{b}+O(N)\right)^{d-1}
      \left(\frac{N^4}{6b} + O(N^3)\right)\\
     &=\frac{N^{2d+2}}{6\cdot\#\Omega} \sum_{b=1}^N   \frac{\mu(b)}{b^d}  
            \left(1 + O\bigg(\frac{b}{N}\bigg)\right)^d\\
     &=\frac{N^{2d+2}}{6\cdot\#\Omega} \left(\frac{1}{\zeta(d)}+O\bigg(\frac{1}{dN^{d-1}}\bigg)
     + \sum_{b=1}^N \frac{1}{b^d} \sum_{k=1}^d \binom{d}{k}\bigg(\frac{b}{N}\bigg)^k\right).
    \end{split}
\end{equation} 
Here the interior sums are
\begin{equation*}
    \begin{split}
     \sum_{b=1}^N \frac{1}{b^d} \sum_{k=1}^d \binom{d}{k}\bigg(\frac{b}{N}\bigg)^k
     =\begin{cases}
         O\left(\frac{\log N}{N}\right)& \text{ if $d=2$},\\[8pt]
         O\left(\frac{d}{N}\right)& \text{ if $d\ge 3$}.
      \end{cases}
    \end{split}
\end{equation*} 
Introducing this estimate in~\eqref{eqAv16} and the result in~\eqref{eqAv11} we summarize 
in the next lemma the estimate obtained for $A_{vis}(d,N)$. 
\begin{lemma}\label{LemmaAvisOmega}
   We have
   \begin{equation}\label{eqAisEnd1}
     A_{vis}(d,N)= 
     \begin{cases}
               \frac{dN^{2d+2}}{6\#\Omega\zeta(d)}
        \left(1+O\left(\frac {\log N}{N}\right)\right) & \text{ if $d= 2$},\\[8pt]
         \frac{dN^{2d+2}}{6\#\Omega\zeta(d)}\left(
         1+O\left(\frac dN\right)\right) & \text{ if $d\ge 3$}.
     \end{cases}
\end{equation}
\end{lemma}

Taking into account the size of the cardinality of 
$\Omega$ evaluated in Lemma~\ref{LemmaEstimateOmega} 
into~\eqref{eqAisEnd1}, it yields
the following simple estimate for $A_{vis}(d,N)$.
\begin{lemma}\label{LemmaEstimateAvis}
There exists an absolute constant $C_1>0$, such that for all $d\ge 2$  and all $N\ge 3d$, we have
\begin{equation}\label{eqAvis}
    \left|
    A_{vis}(d,N)- \frac{dN^{2}}{6}
    \right|\le
    \begin{cases}
       C_1N\log N & \text{ if $d=2$}, \\[8pt]
       C_1d^2N & \text{ if $d\ge 3$}.
    \end{cases}
\end{equation}
\end{lemma}

\section{The second moment about the mean}\label{SectionM2}
\noindent
The second moment about the mean $A_{vis}(d,N)$ is the average of the squares of the differences 
between the expected and the true distance between the pairs of points from $\cW$ 
that are visible from each other, that is,
\begin{equation}\label{eqM221}
    \begin{split}
      \fM_{2,vis}(d,N):=\frac{1}{\#\Omega}
    \sum_{\bv\in\cW}\sum_{\substack{\bw\in\cW\\(\bv,\bw)\in\Omega}}
    \left|\distance^2(\bv,\bw)-A_{vis}(d,N)\right|^2.
    \end{split}
\end{equation}
Replacing the coprimality condition by means of the characteristic function~\eqref{eqMBS} 
and changing the order of summation, we have
\begin{equation*}
       \begin{split}
      \fM_{2,vis}(d,N)=&\frac{1}{\#\Omega}
      \sum_{b=1}^N\mu(b)
      \sum_{\bv\in\cW}\sum_{\substack{\bw\in\cW\\b|(v_1-w_1)\\
        \cdots\\
    b|(v_d-w_d)
    }}
    \left|\distance^2(\bv,\bw)-A_{vis}(d,N)\right|^2.
    \end{split}
\end{equation*}
Next, by expanding the square it yields
\begin{equation}\label{eqM2211}
    \begin{split}
      \fM_{2,vis}(d,N)   
    =& \frac{1}{\#\Omega}
    \sum_{b=1}^N\mu(b)
      \sum_{\bv\in\cW}\sum_{\substack{\bw\in\cW\\b|(v_1-w_1)\\
        \cdots\\
    b|(v_d-w_d)
    }}\distance^4(\bv,\bw)\\
    & -\frac{2A_{vis}(d,N)}{\#\Omega}
    \sum_{b=1}^N\mu(b)
      \sum_{\bv\in\cW}\sum_{\substack{\bw\in\cW\\b|(v_1-w_1)\\
        \cdots\\
    b|(v_d-w_d)
    }}\distance^2(\bv,\bw)\\
    & + \frac{A^2_{vis}(d,N)}{\#\Omega}
     \sum_{b=1}^N\mu(b)
      \sum_{\bv\in\cW}\sum_{\substack{\bw\in\cW\\b|(v_1-w_1)\\
        \cdots\\
    b|(v_d-w_d)
    }}1\\
    =& \frac{1}{\#\Omega}\cdot\Sigma_{vis} - A^2_{vis}(d,N).
    \end{split}
\end{equation}
Here we have denoted by $\Sigma_{vis}$ the multiple sum over $\bv$ and $\bw$ from the first row of 
relation~\eqref{eqM2211} and have taken into account the fact that the term from the second row 
is equal to $-2A^2_{vis}(d,N)$, while the term from the third row is equal to  $A^2_{vis}(d,N)$. 
Next, changing the order of summation, we split~$\Sigma_{vis}$ into $d^2$ similar sums $H_{j,k}$
\begin{equation}\label{eqM2a22}
    \begin{split}
      \Sigma_{vis} &= \sum_{b=1}^N\mu(b)\sum_{\bv\in\cW}
            \sum_{\substack{\bw\in\cW\\b|(v_1-w_1)\\\cdots\\b|(v_d-w_d)}}
      \sum_{j=1}^d\sum_{k=1}^d(w_j-v_j)^2(w_k-v_k)^2
      = \sum_{j=1}^d\sum_{k=1}^d H_{j,k},
    \end{split}
\end{equation}
where 
\begin{equation}\label{eqM2a23}
    \begin{split}
    H_{j,k}  &= \sum_{b=1}^N\mu(b)\sum_{\bv\in\cW}
            \sum_{\substack{\bw\in\cW\\b|(v_1-w_1)\\\cdots\\b|(v_d-w_d)}}
                                (w_j-v_j)^2(w_k-v_k)^2.                    
    \end{split}
\end{equation}
Now fix $j\neq k$.
In each $H_{j,k}$ the summand depends only on four of the $2d$ variables $v_1,\dots,v_d,w_1,\dots,w_d$,
so that
\begin{equation}\label{eqM2a24}
    \begin{split}
    H_{j,k} &= \sum_{b=1}^N\mu(b)\bigg(
        \sum_{\substack{0\le v_j,w_j\le N\\b|(v_j-w_j)}}
        \sum_{\substack{0\le v_k,w_k\le N\\b|(v_k-w_k)}}
                                (w_j-v_j)^2(w_k-v_k)^2\bigg)    
             \prod_{\substack{s=1\\s\not=j\\s\not=k}}^d
            \bigg(\sum_{\substack{0\le v_s,w_s\le N\\b|(v_s-w_s)     }} 1\bigg).
     \end{split}
\end{equation}
Since the products count the number of terms in some arithmetic progressions and are equal, 
we derive that
\begin{equation}\label{eqM2a25}
    \begin{split}
    H_{j,k} &= \sum_{b=1}^N\mu(b)  \left(\frac{N^2}{b}+O(N)\right)^{d-2}
     \sum_{\substack{0\le v_j,w_j\le N\\b|(v_j-w_j)}}
        \sum_{\substack{0\le v_k,w_k\le N\\b|(v_k-w_k)}}
                                (w_j-v_j)^2(w_k-v_k)^2                
    \end{split}
\end{equation}
By relation~\eqref{eqAv155}, we find that the interior sums are 
\begin{equation}\label{eqM2a255}
    \begin{split}
         \sum_{\substack{0\le v_j,w_j\le N\\b|(v_j-w_j)}}
        \sum_{\substack{0\le v_k,w_k\le N\\b|(v_k-w_k)}}
                                (w_j-v_j)^2(w_k-v_k)^2   
        =\frac{N^8}{36b^2} +O\left(N^7/b\right).                        
    \end{split}
\end{equation}
On combining~\eqref{eqM2a255} and~\eqref{eqM2a25}, it follows that
\begin{equation}\label{eqM2a26}
    \begin{split}
    H_{j,k} &= \sum_{b=1}^N\mu(b)  \left(\frac{N^2}{b}+O(N)\right)^{d-2}
     \left(\frac{N^8}{36b^2} +O\left(N^7/b\right)\right)\\
     &=\frac{N^{2d+4}}{36} \sum_{b=1}^N   \frac{\mu(b)}{b^{d}}  
            \left(1 + O\bigg(\frac{b}{N}\bigg)\right)^{d}.
    \end{split}
\end{equation}
Following the reasoning from~\eqref{eqAv16} and the relation that follows, we obtain the following estimate
\begin{equation}\label{eqM2a27}
    \begin{split}
    H_{j,k} =\begin{cases}
    \frac{N^{2d+4}}{36\zeta(d)} \left(1+
         O\left(\frac{\log N}{N}\right)\right)& \text{ if $d=2$},\\[8pt]
    \frac{N^{2d+4}}{36\zeta(d)}  \left(1+
         O\left(\frac{d}{N}\right)\right)& \text{ if $d\ge 3$}.
      \end{cases}
    \end{split}
\end{equation}

If $j=k$, adapting the same steps after relation~\eqref{eqM2a23} we obtain the upper bound
\begin{equation}\label{eqM2a28}
    \begin{split}
    H_{j,j} =O(N^{2d+4}), \text{ for $1\le j \le d$.}
    \end{split}
\end{equation}
Then on inserting~\eqref{eqM2a28} and~\eqref{eqM2a27} into~\eqref{eqM2a22}, yields
\begin{equation}\label{eqM2a30}
    \begin{split}
    \Sigma_{vis}= \frac{d^2N^{2d+4}}{36\zeta(d)} 
    \left(1+O\left(\frac{1}{d}+\frac{d}{N}\right)\right).
    \end{split}
\end{equation}
On combining~\eqref{eqM2a30}, \eqref{eqM2211}, Lemma~\ref{LemmaEstimateOmega}
and Lemma~\ref{LemmaEstimateAvis}, we obtain the following result.

\begin{lemma}\label{LemmaEstimateM2vis}
There exists an absolute constant $C_2>0$, such that for all $d\ge 2$  and all $N\ge 3d$, we have
\begin{equation}\label{eqAis2}
   \fM_{2,vis}(d,N)\le C_2(dN^4 +d^3N^3).
\end{equation}
\end{lemma}

\section{Effective results and the proofs of Theorems~\ref{Theorem1}, \ref{Theorem2} and Corollary~\ref{Corollary1}}\label{SectionProofT1T2C1}

%
We scale the bound for $\mathfrak{M}_{2,vis}(d,N)$ from Lemma~\ref{LemmaEstimateM2vis} by $d^2N^4$, 
in order to have all spacings between points measured by the normalized distance 
situated in the interval $[0,1]$.
Note first that for any $d\ge 2$ and any $N\ge 3d$, we have 
\begin{equation*}
    \frac{\mathfrak{M}_{2,vis}(d,N)}{d^2N^4}\le C_2\left(\frac 1d +\frac dN\right).
\end{equation*}
Then, on combining the above inequalities with Lemma~\ref{LemmaEstimateAvis}, 
there is an absolute constant $C_3>0$ such that
\begin{equation}\label{eqPT1}
    \frac{1}{\#\Omega}
    \sum_{(\bv,\bw)\in\Omega}\left(\distance_d^2(\bv,\bw)-\sdfrac{1}{6}\right)^2
    \le  C_3\left(\frac 1d +\frac dN\right).
\end{equation}
Now, for any parameters $a,T>0$, imposing supplementary conditions on the summation, we find the following lower bounds of the left-side term of the inequality~\eqref{eqPT1}:
\begin{equation}\label{eqPT12}
    \begin{split}
    \frac{1}{\#\Omega} \sum_{(\bv,\bw)\in\Omega}
            \left(\distance_d^2(\bv,\bw)-\sdfrac{1}{6}\right)^2
    &\ge  \frac{1}{\#\Omega} 
      \sum_{\substack{(\bv,\bw)\in\Omega\\
        \big|\distance_d^2(\bv,\bw)-\frac{1}{6}
        \big|\ge \frac 1{aT}}}
            \left(\distance_d^2(\bv,\bw)-\sdfrac{1}{6}\right)^2 \\
      &\ge  \frac{1}{\#\Omega} 
      \sum_{\substack{(\bv,\bw)\in\Omega\\
        \big|\distance_d^2(\bv,\bw)-\frac{1}{6}
        \big|\ge \frac 1{aT}}}\frac{1}{a^2T^2}.
    \end{split}
\end{equation}
Then, on combining~\eqref{eqPT1} and~\eqref{eqPT12}, we find that 
\begin{equation}\label{eqPT2}
    \frac{1}{\#\Omega}\#\left\{
       (\bv,\bw) \in\Omega \; :\; \left| \distance_d^2(\bv,\bw)-\frac{1}{6}\right|
       \ge \frac{1}{aT}
       \right\}
       \le  C_3a^2T^2\left(\frac 1d +\frac dN\right).
\end{equation}
Now, since
\begin{equation*}
       \left| \distance_d^2(\bv,\bw)-\frac{1}{6}\right| = 
 \left| \distance_d(\bv,\bw)-\frac{1}{\sqrt{6}}\right|
 \left( \distance_d(\bv,\bw)+\frac{1}{\sqrt{6}}\right)
    \ge \frac{1}{\sqrt{6}}\left| \distance_d(\bv,\bw)-\frac{1}{\sqrt{6}}\right|,
\end{equation*}
by sharpening the restriction in the definition of the set on the left side of~\eqref{eqPT2}, the set remains with fewer elements, so that
with $a=\sqrt{6}$, we derive that
\begin{equation}\label{eqPT3}
    \frac{1}{\#\Omega}\#\left\{
       (\bv,\bw) \in\Omega \; :\; \left| \distance_d(\bv,\bw)-\frac{1}{\sqrt{6}}\right|\ge \frac{1}{T}
       \right\}
       \le  6C_3T^2\left(\frac 1d +\frac dN\right).
\end{equation}
In particular, this proves Theorem~\ref{Theorem1}.

\bigskip 
More generally, we consider the set $\Omega_K$ of $K$-polytopes 
$P=\{\bw_1,\dots,\bw_K\}\subset\cW$
the property that any of its two vertices are visible from each other.
Then Corollary~\ref{Corollary1} follows from the following 
more general statements.


\begin{theorem}\label{TheoremK11}
There exists an effectively computable absolute constant $C_4>0$ such that for any integers
$d\ge 2$, $N\ge 3d$, $K\ge 2$ and any real $T>0$, we have 
\begin{equation*}\label{eqPTK1}
    \sdfrac{1}{\#\Omega_K}\cdot 
        \#\left\{\{\bw_1,\dots,\bw_K\}\subset \Omega_K \; :\;
        \max_{1\le m\not = n\le K}\left|\distance_d(\bw_m,\bw_n)-\sdfrac{1}{\sqrt{6}}\right|\ge \sdfrac{1}{T}
       \right\}
                    \le C_4T^2K^2\left(\frac {1}{d} +\frac{d}{N}\right).
    \end{equation*}
\end{theorem}

\begin{corollary}\label{CorollaryK11}
Let $\eta\in (0,1/4)$ be fixed.
Then, there exists an effectively computable absolute constant $C_5>0$ such that
for any integers $d\ge 2$, $N\ge d^2$, $2 \le K\le d^{1/2-2\eta}$, we have 
\begin{equation*}
   \sdfrac{1}{\#\Omega_K}\cdot \#\left\{\{\bw_1,\dots,\bw_K\}\subset \Omega_K \; :\;
    \left.\begin{aligned}
        & \distance_d(\bw_m,\bw_n)\in \left[\sdfrac{1}{\sqrt{6}}-\sdfrac{1}{d^\eta},
                                    \sdfrac{1}{\sqrt{6}}+\sdfrac{1}{d^\eta}    \right]\\
        & \text{ for all $1\le m\not = n\le K$ }
    \end{aligned}
    \right.
\right\}
\ge 1- \sdfrac{C_5}{d^{2\eta}}.
\end{equation*}
\end{corollary}

For the proof of Theorem~\ref{Theorem2} one can follow the path from Sections~\ref{SectionAv} 
and~\ref{SectionM2} with one component $\bv=\mathbf{0}$ fixed in the involved summations.
One finds that almost all normalized distances between the origin and the components of points in $\Omega$ are close to $1/\sqrt{3}$. 

On the other hand, we know that, according 
to Theorem~\ref{Theorem1}, almost all normalized distances between $\bv$ and $\bw$ 
with $(\bv,\bw)\in\Omega$ are almost always almost equal to $1/\sqrt{6}$.
Therefore, almost all triangles with vertices $\mathbf{0}, \bv, \bw$ with $(\bv,\bw)\in\Omega$
are almost isosceles having the normalized edges almost equal to $1/\sqrt{3},1/\sqrt{3},1/\sqrt{6}$ 
and Theorem~\ref{Theorem2} follows immediately.

\section{The probability that a \texorpdfstring{$K$}{K}-polytope is Self-Visible}\label{SectionProbabilityOK}

Let $K\ge 2$ be a fixed integer. The set of self-visible $K$-polytopes with
vertices in the lattice $\cW$ is
\begin{equation}\label{eqOmegaKdef}
    \Omega_K=\left\{
    P\in \cW^K \; :\; \bv',\bv'' \text{ visible from each other, for all }
    \bv',\bv''\in P
    \right\}.
\end{equation}

Our object here is to see if there is a tendency of 
the probabilities that a $K$-polytope 
is self-visible as $N$ gets large. We show that if $d$ and $K$ are kept fixed,
the limit of the ratios
\begin{equation*}
   Prob(d,N,K)=\lim_{N\to\infty}\frac{\#\Omega_K}{\#\cW^K}
\end{equation*}
does exist.

If $K=2$, then $\Omega_K$ coincides with $\Omega$, but the M\"obius 
summation method used in the proof of Lemma~\ref{LemmaEstimateOmega} 
to estimate $\#\Omega_2$ is not suitable for larger $K$, because of 
the size of the multitude of new terms introduced.
We need to have a better control on the large divisors, 
so we will proceed accordingly.

Denote a generic polytope by $P=\{\bv_1,\dots,\bv_K\}$ 
and the coordinates of its vertices by
$\bv_j=(v_{j,1},\dots,v_{j,d})$ for $1\le j\le K$.
Note that, for each positive integer $m$, we have the following inequality
\begin{equation}\label{eqOmegaBoundjk}
    \#\left\{(\bv_j,\bv_k)\in\cW^2 \; :\; 
    \bv_j\neq \bv_k,\ m\mid \gcd(v_{j,1}-v_{k,1},\dots,v_{j,d}-v_{k,d})
    \right\}\le \frac{\#\cW^2}{m^d},
\end{equation}
because, say,  $v_{j,1},\dots,v_{j,d}$ 
are free and then each of $v_{k,1},\dots,v_{k,d}$
belongs to the corresponding shifted arithmetic progression of ratio $m$.
Also, if $m\ge N$, the left side of~\eqref{eqOmegaBoundjk} 
equals zero, since there are no pairs to count.

Fix $M>0$, a parameter to be chosen later, and sum the inequalities~\eqref{eqOmegaBoundjk}
for all~\mbox{$m>M$}. Then the size of the resulted sum is 
\begin{equation*}
   \le  \sum_{m> M}\frac{\#\cW^2}{m^d}=O\left(\frac{\#\cW^2}{M^{d-1}}\right).
\end{equation*}
As a consequence, any such subsum is also $\ll \#\cW^2/M^{d-1}$.
In particular, the sum over all positive integers $m$ that have 
at least one prime factor larger than $M$. 
This holds for each pair $(\bv_j,\bv_k)$, and there are $K(K-1)/2$ 
such pairs with $1\le j<k\le K$.
As a consequence, it follows that
\begin{equation}\label{eqOmegaBoundq}
    \#\left\{(\bv_1,\dots,\bv_K)\in \cW^K\; :\; 
    \left.\begin{aligned}
        & q\mid \gcd(v_{j,1}-v_{k,1},\dots,v_{j,d}-v_{k,d})\\
        & \text{ for some prime $q>M$,}\\
        & \text{ for some $1\le j<k\le K$ }
    \end{aligned}
    \right.
        \right\}
    \ll \frac{K^2\#\cW^K}{M^{d-1}}.
\end{equation}
In other words, with the exception of at most $O\left(\frac{K^2\#\cW^K}{M^{d-1}}\right)$
$K$-tuples $(\bv_1,\dots,\bv_K)$,
for all the other polytopes $P=(\bv_1,\dots,\bv_K)\in\cW^K$, the condition
$P\in\Omega_K$ is equivalent to the condition that $P\in\Omega_{K}(M)$, where
\begin{equation}\label{eqOmegaKM}
    \Omega_K(M):=
    \left\{(\bv_1,\dots,\bv_K)\in \cW^K\; :\; 
    \left.\begin{aligned}
        &  \gcd(B,v_{j,1}-v_{k,1},\dots,v_{j,d}-v_{k,d})=1\\
        & \text{ for all $1\le j<k\le K$ }
    \end{aligned}
    \right.
        \right\},
\end{equation}
where $B$ is the primorial number
\begin{equation*}
   B:=\prod\limits_{\substack{\text{$p$ prime}\\ p\le M}}p.
\end{equation*}
Therefore, the probability that a polytope $P\in\cW^K$ 
has all vertices visible from each other is
\begin{equation}\label{eqOmegaOmegaM}
   \frac{\#\Omega_K}{\#\cW^K}=\frac{\#\Omega_K(M)}{\#\cW^K}
   + O\left(\frac{K^2\#\cW^K}{M^{d-1}}\right).
\end{equation}
By the Prime Number Theorem, we know that $B=e^{(1+o(1))M}$, so that we will eventually choose $M$ of 
size $\log N$ to assure that $B<N$.

Next, we split the interval $[0,N]$ in subintervals of size $B$.
Accordingly, the cube $[0,N]^d$ is split in boxes of side length $B$.
The number of these boxes~is
\begin{equation}\label{eqNumberOfBoxes}
    \text{The number of boxes} = \left(\frac{N}{B}+O(1)\right)^d
    =\frac{N^d}{B^d}+O\left(\frac{dN^{d-1}}{B^{d-1}}\right).
\end{equation}
Observe, by the definition, that $\Omega_K(M)$ has the same number of elements in each such box.
Denote this number by $H(B)$, that is,
\begin{equation}\label{eqHB}
    H(B):=\#
    \left\{(\bv_1,\dots,\bv_K)\in \Omega_K(M)\; :\; 
          0\le v_{j,l}<B \text{ for all $1\le j\le K$, $1\le l\le d$}
        \right\}.
\end{equation}
Then, by~\eqref{eqNumberOfBoxes} and~\eqref{eqHB}, as each $\bv_1,\dots,\bv_K$ runs over each box,
it follows that
\begin{equation}\label{eqSizeOfOmegaKM}
    \#\Omega_K(M)=H(B)\left(\frac{N^d}{B^d}
    +O\left(\frac{dN^{d-1}}{B^{d-1}}\right)\right)^K
    =\frac{N^{dK}H(B)}{B^{dK}}
    \left(1+O\left(\frac{dKB}{N}\right)\right).
\end{equation}
Since $H(B)\le B^{dK}$ and since $\#\cW^K=(N+1)^{dK}=N^{dK}(1+O(dK/N))$, 
it follows that
\begin{equation}\label{eqProbability102}
   \frac{\#\Omega_K(M)}{\#\cW^{K}} = \frac{H(B)}{B^{dK}}\left(1+O\left(\frac{dKB}{N}\right)\right)
   =  \frac{H(B)}{B^{dK}}+O\left(\frac{dKB}{N}\right).
\end{equation}
On combining~\eqref{eqSizeOfOmegaKM} and~\eqref{eqOmegaOmegaM}, it yields 
\begin{equation}\label{eqeqProbability103}
   \frac{\#\Omega_K}{\#\cW^{K}} = \frac{H(B)}{B^{dK}}+O\left(\frac{dKB}{N}\right)
   + O\left(\frac{K^2}{M^{d-1}}\right).
\end{equation}

Now, for each prime $p\mid B$, consider the analogue of 
the set $\Omega_K(M)$ defined 
by~\eqref{eqOmegaKM}. Its cardinality is analogous to $H(B)$ and is given by
\begin{equation}\label{eqHdep}
    H(p):=
    \#\left\{(\bv_1,\dots,\bv_K)\in \cW^K :
    \left.\begin{aligned}
        &  \gcd(p,v_{j,1}-v_{k,1},\dots,v_{j,d}-v_{k,d})=1\\
        & \text{ for all $1\le j<k\le K$, }\\
        &  0\le v_{j,l} \le p-1  \text{ for all $1\le j\le K$, $1\le l\le d$}
    \end{aligned}
    \right.
        \right\}.
\end{equation}
Note that each $(\bv_1,\dots,\bv_K)\in \cW^K$ that contributes to $H(B)$ produces,
via reduction modulo $p$, a $K$-tuple that contributes to $H(p)$, 
and this holds for each prime divisor $p$ of $B$.
Conversely, by the Chinese Remainder Theorem, each collection of 
$K$-tuples, with one $K$-tuple for
each prime divisor of $B$, produces a unique $K$-tuple that is counted in $H(B)$.
In conclusion, 
\begin{equation*}
   H(B)=\prod_{\substack{\text{$p$ prime}\\ p\mid B}} H(p)
   = \prod_{\substack{\text{$p$ prime}\\ p\le M}} H(p),
\end{equation*}
which combined with~\eqref{eqeqProbability103} implies
\begin{equation}\label{eqProbability104}
    \frac{\#\Omega_K}{\#\cW^{K}} = \prod_{\substack{\text{$p$ prime}\\ p\le M}} \frac{H(p)}{p^{dK}}
    +O\left(\frac{dKB}{N}\right)
   + O\left(\frac{K^2}{M^{d-1}}\right).
\end{equation}

Next, let us observe that since each of the coordinates 
$v_{j,1},\dots,v_{j,d}$ and $v_{k,1},\dots,v_{k,d}$
belongs to $\{0,1,\dots,p-1\}$, the difference
$v_{j,1}-v_{k,1}$ cannot be divisible by $p$ unless $v_{j,1}=v_{k,1}$, 
and similarly for all differences $v_{j,2}-v_{k,2}, \dots, v_{j,d}-v_{k,d}$.
As a consequence, the condition
$\gcd(p,v_{j,1}-v_{k,1},\dots,v_{j,d}-v_{k,d})=1$ from the definition of $H(p)$ given 
by~\eqref{eqHdep} is equivalent to the condition that the $d$-tuples
$(v_{j,1},\dots,v_{j,d})$ and $(v_{k,1},\dots,v_{k,d})$ are distinct.
In other words
\begin{equation*}
    H(p)=
    \#\left\{(\bv_1,\dots,\bv_K)\in \cW^K\; :\; 
    \left.\begin{aligned}
        &  \mbox{}\bv_{j}\neq\bv_k \text{ for $1\le j\neq k\le K$,}\\
        &  0\le v_{j,l} \le p-1  \text{ for all $1\le j\le K$ and $1\le l\le d$ }
    \end{aligned}
    \right.
        \right\}.
\end{equation*}
Here, there are exactly $p^d$ choices for $\bv_1$.
Then, for each fixed $\bv_1$, the only restriction on~$\bv_2$ 
is to not coincide with $\bv_1$, so that there are $p^d-1$ choices for $\bv_2$.
With $\bv_1$ and $\bv_2$ fixed, the only restrictions 
on $\bv_3$ are $\bv_3\neq\bv_1$ and 
$\bv_3\neq \bv_2$, so that there are $p^{d}-2$ choices for $\bv_3$.
And so on, up to $\bv_K$, for which there are $p^d-(K-1)$ choices.
In conclusion 
\begin{equation*}
   H(p)= p^d\left(p^d-1\right)\cdots \left(p^d-(K-1)\right).
\end{equation*}
On combining this with~\eqref{eqProbability104} we see that
\begin{equation}\label{eqProbability105}
    \frac{\#\Omega_K}{\#\cW^{K}} = \prod_{\substack{\text{$p$ prime}\\ p\le M}}
    \left(1-\frac{1}{p^d}\right)\cdots
    \left(1-\frac{K-1}{p^d}\right)
    +O\left(\frac{dKB}{N}\right)
   + O\left(\frac{K^2}{M^{d-1}}\right).
\end{equation}
The finite product over primes in~\eqref{eqProbability105} 
can be replaced with the completed product over all primes, 
with a change in the error term that is swallowed inside 
the last error term. Indeed, if we denote
\begin{equation*}
  \Lambda_{d,K}(M) :=\prod_{\substack{\text{$p$ prime}\\ p> M}}
  \left(1-\frac{1}{p^d}\right)\cdots
    \left(1-\frac{K-1}{p^d}\right),
\end{equation*}
an infinite product that converges if $d\ge 2$, then
\begin{equation*}
  \log \Lambda_{d,K}(M) =\sum_{\substack{\text{$p$ prime}\\ p> M}}
  \sum_{1\le k\le K-1}\log \left(1-\frac{k}{p^d}\right)
  =\sum_{\substack{\text{$p$ prime}\\ p> M}}
  \sum_{1\le k\le K-1}O\left(\frac{k}{p^d}\right).
\end{equation*}
This implies
\begin{equation*}
    \begin{split}
  |\log \Lambda_{d,K}(M)| &=O\bigg(\sum_{\substack{\text{$p$ prime}\\ p> M}}
  \sum_{1\le k\le K-1}\frac{k}{p^d}
  \bigg)
  = O\bigg(\sum_{\substack{\text{$p$ prime}\\ p> M}}
  \frac{K^2}{p^d} \bigg)
  =O\bigg(\sum_{m> M}
  \frac{K^2}{m^d} \bigg)
  =O\left(  \frac{K^2}{M^{d-1}} \right).
  \end{split}
\end{equation*}  
It follows that
\begin{equation*}
    \Lambda_{d,K}(M) = \exp\left(O\left(\sdfrac{K^2}{M^{d-1}}\right)\right)
     = 1+O\left(\sdfrac{K^2}{M^{d-1}}\right).
\end{equation*}
Therefore, if we denote by $\Lambda_{d,K}$ the complete infinite product, 
\begin{equation}\label{eqCompleteProduct}
    \Lambda_{d,K}:= \prod_{\text{$p$ prime}}
    \left(1-\sdfrac{1}{p^d}\right)\cdots
    \left(1-\sdfrac{K-1}{p^d}\right),
\end{equation}    
which is constant for any fixed $d$ and $K$, we have
\begin{equation}\label{eqEstimateProduct}
    \begin{split}
  \prod_{\substack{\text{$p$ prime}\\ p\le M}}
    \left(1-\sdfrac{1}{p^d}\right)\cdots
    \left(1-\sdfrac{K-1}{p^d}\right) 
    &=  \frac{\Lambda_{d,K}}{\Lambda_{d,K}(M)}
    = \Lambda_{d,K}
    + O\left(\sdfrac{K^2}{M^{d-1}}\right),
    \end{split}
\end{equation}
where the implied constant in the big $O$ estimate is absolute,
because $\Lambda_{d,2}=\zeta(d)^{-1}$ for $d\ge 2$ and,
for any fixed $d$, the sequence $\{\Lambda_{d,K}\}_{K\ge 2}$ is
decreasing.

Then, inserting~\eqref{eqEstimateProduct} and~\eqref{eqCompleteProduct} in~\eqref{eqProbability105}, we arrive at the following result
\begin{equation*}
    \frac{\#\Omega_K}{\#\cW^{K}} = \prod_{\substack{\text{$p$ prime}}}
    \left(1-\frac{1}{p^d}\right)\cdots
    \left(1-\frac{K-1}{p^d}\right)
    +O\left(\frac{dKB}{N}\right)
   + O\left(\frac{K^2}{M^{d-1}}\right).
\end{equation*}

We now take $B$ to be the largest primorial that is~$\le \sqrt{N}$, 
which means that \mbox{$M\sim (\log N)/2$}.
Then,  
\begin{equation}\label{eqProbability107}
    \frac{\#\Omega_K}{\#\cW^{K}} = \prod_{\substack{\text{$p$ prime}}}
    \left(1-\frac{1}{p^d}\right)\cdots
    \left(1-\frac{K-1}{p^d}\right)
    +O\left(\frac{dK}{\sqrt{N}}\right)
   + O\left(\frac{2^dK^2}{\log^{d-1}N}\right),
\end{equation}
and the constants implied in the big $O$ terms are absolute.
This concludes the proof of Theorem~\ref{TheoremSelfVisible}.

\section{Probabilistic Intuition}\label{SectionIntuition}
In this section, we show how to interpret the constant $1/\sqrt{6}$ in
Theorem~\ref{Theorem1} via probabilistic intuition. Similar arguments can
give intuition for some of the other particular constants we obtain in this
paper. We recall

\begin{equation*}
\distance_d(\bv,\bw)=\frac{1}{Nd^{1/2}}
    \bigg(\sum_{n=1}^d (w_n-v_n)^2\bigg)^{1/2} = \left( \frac{1}{d} 
    \sum_{n=1}^d \left(\frac{w_n}{N}-\frac{v_n}{N}\right)^2 \right)^{1/2},
\end{equation*}    
As $N \rightarrow \infty$, if we select $\bv$ uniformly at random from
$\mathcal W$ (or, indeed, the subset of integer vectors visible from the
origin in~$\cW$), the normalized vector $\bv/N$ becomes equidistributed 
in the hypercube $[0, 1]^d$. That is, 
\begin{equation*}
    \sum_{\bv \in \mathcal W} \delta_{\bv/N} \rightarrow m,
\end{equation*}
where $m$ is the standard Lebesgue measure on $[0, 1]^d$, 
and the convergence is
in the weak*-sense as $N \rightarrow \infty$. 
The same result is true with $\mathcal W$ replaced by the subset of primitive
vectors in $\mathcal W$. 
Thus, to try and get intuition about the $d \rightarrow \infty$ 
behavior of our normalized distance~$\distance_d(\bv,\bw)$, 
we can consider the following probabilistic analogue. 
Let $\mathbf{X} = (X_1, \ldots X_d), \mathbf{Y} = (Y_1, \ldots, Y_d)$ 
be independent random vectors chosen according to 
Lebesgue measure on the hypercube $[0, 1]^d$. 
Thus $X_1, \ldots, X_d$ and $Y_1, \ldots Y_d$ are independent, 
identically distributed (i.i.d.) uniform $[0,1]$ random variables, 
and also indpendent from each other. 
We define 
\begin{equation*}
    \distance_d(\mathbf X,\mathbf Y) 
    = \left( \frac{1}{d} \sum_{n=1}^d \left(X_n - Y_n\right)^2 \right)^{1/2}.
\end{equation*}
Thus 
\begin{equation*}
    \distance_d^2(\mathbf X,\mathbf Y) 
    =  \frac{1}{d} \sum_{n=1}^d \left(X_n - Y_n\right)^2
\end{equation*}
is the sample mean of $d$ independent random variables of the form 
$(U-V)^2$, where $U$ and~$V$ are independent uniform $[0, 1]$ 
random variables. 
By the strong law of large numbers, as~$d \rightarrow \infty$, 
this converges with probability $1$ to the mean 
\begin{equation*}
    E( (U-V)^2) =  \int_0^1 \int_0^1 (u-v)^2 du dv = 1/6.
\end{equation*}
That is, as $d \rightarrow \infty,$ with probability $1$,
\begin{equation*}
   \distance_d^2(\mathbf X,\mathbf Y) \rightarrow 1/6,
\end{equation*}
so 
\begin{equation*}
   \distance_d (\mathbf X,\mathbf Y) \rightarrow 1/\sqrt{6}.
\end{equation*}
To be clear, this does not give a direct proof of 
Theorem~\ref{Theorem1}, since there are tricky issues with 
the interchange of limits. 
Similar arguments can yield intuition for the other constants 
in our results.

\subsection*{Acknowledgement}
We thank Sara Billey, Sam Fairchild, Alex Kontorovich, and Doug West for valuable discussions at a variety of times about the problem of the limiting density of  $\Omega_K$. J.S.A. was partially supported by NSF grant DMS 2003528, `Curves, Counting, and Correlations'. J.S.A. also acknowledges the hospitality of the Mathematical Sciences Research Institute (MSRI) during the Spring 2022 program on Analysis and Geometry of Random Spaces.



\end{document}